\documentclass[reqno, 12pt]{amsart}
\pdfoutput=1
\makeatletter
\let\origsection=\section \def\section{\@ifstar{\origsection*}{\mysection}}
\def\mysection{\@startsection{section}{1}\z@{.7\linespacing\@plus\linespacing}{.5\linespacing}{\normalfont\scshape\centering\S}}
\makeatother

\usepackage{amsmath,amssymb,amsthm}
\usepackage{mathrsfs}
\usepackage{mathabx}\changenotsign
\usepackage{dsfont}

\usepackage{xcolor}
\usepackage[backref]{hyperref}
\hypersetup{
    colorlinks,
    linkcolor={red!60!black},
    citecolor={green!60!black},
    urlcolor={blue!60!black}
}

\usepackage[open,openlevel=2,atend]{bookmark}

\usepackage[abbrev,msc-links,backrefs]{amsrefs}
\usepackage{doi}

\renewcommand{\PrintDOI}[1]{\doi{#1}}

\usepackage[T1]{fontenc}
\usepackage{lmodern}
\usepackage[babel]{microtype}
\usepackage[english]{babel}

\usepackage{geometry}
\numberwithin{equation}{section}
\numberwithin{figure}{section}

\usepackage{enumitem}
\def\rmlabel{\upshape({\itshape \roman*\,})}

\def\alabel{\upshape({\itshape \alph*\,})}

\def\nlabel{\upshape({\itshape \arabic*\,})}

\let\polishlcross=\l
\def\l{\ifmmode\ell\else\polishlcross\fi}

\def\qand{\quad\text{and}\quad}
\def\qqand{\qquad\text{and}\qquad}

\let\emptyset=\varnothing
\let\setminus=\smallsetminus

\makeatletter
\def\moverlay{\mathpalette\mov@rlay}
\def\mov@rlay#1#2{\leavevmode\vtop{   \baselineskip\z@skip \lineskiplimit-\maxdimen
   \ialign{\hfil$\m@th#1##$\hfil\cr#2\crcr}}}
\newcommand{\charfusion}[3][\mathord]{
    #1{\ifx#1\mathop\vphantom{#2}\fi
        \mathpalette\mov@rlay{#2\cr#3}
      }
    \ifx#1\mathop\expandafter\displaylimits\fi}
\makeatother

\newcommand{\dcup}{\charfusion[\mathbin]{\cup}{\cdot}}
\newcommand{\bigdcup}{\charfusion[\mathop]{\bigcup}{\cdot}}

\makeatletter
\DeclareFontFamily{U}  {MnSymbolC}{}
\DeclareSymbolFont{MnSyC}         {U}  {MnSymbolC}{m}{n}
\DeclareFontShape{U}{MnSymbolC}{m}{n}{
    <-6>  MnSymbolC5
   <6-7>  MnSymbolC6
   <7-8>  MnSymbolC7
   <8-9>  MnSymbolC8
   <9-10> MnSymbolC9
  <10-12> MnSymbolC10
  <12->   MnSymbolC12}{}
\DeclareMathSymbol{\powerset}{\mathord}{MnSyC}{180}

\DeclareFontFamily{U}{MnSymbolA}{}
\DeclareFontShape{U}{MnSymbolA}{m}{n}{
    <-6>  MnSymbolA5
   <6-7>  MnSymbolA6
   <7-8>  MnSymbolA7
   <8-9>  MnSymbolA8
   <9-10> MnSymbolA9
  <10-12> MnSymbolA10
  <12->   MnSymbolA12}{}
\DeclareFontShape{U}{MnSymbolA}{b}{n}{
    <-6>  MnSymbolA-Bold5
   <6-7>  MnSymbolA-Bold6
   <7-8>  MnSymbolA-Bold7
   <8-9>  MnSymbolA-Bold8
   <9-10> MnSymbolA-Bold9
  <10-12> MnSymbolA-Bold10
  <12->   MnSymbolA-Bold12}{}
\DeclareSymbolFont{MnSyA}{U}{MnSymbolA}{m}{n}
\SetSymbolFont{MnSyA}{bold}{U}{MnSymbolA}{b}{n}

\DeclareRobustCommand{\overleftharpoon}{\mathpalette{\overarrow@\leftharpoonfill@}}
\DeclareRobustCommand{\overrightharpoon}{\mathpalette{\overarrow@\rightharpoonfill@}}
\def\leftharpoonfill@{\arrowfill@\leftharpoondown\mn@relbar\mn@relbar}
\def\rightharpoonfill@{\arrowfill@\mn@relbar\mn@relbar\rightharpoonup}

\DeclareMathSymbol{\leftharpoondown}{\mathrel}{MnSyA}{'112}
\DeclareMathSymbol{\rightharpoonup}{\mathrel}{MnSyA}{'100}
\DeclareMathSymbol{\mn@relbar}{\mathrel}{MnSyA}{'320}

\makeatother

\let\overrightarrow=\overrightharpoon

\usepackage{tikz}
\usetikzlibrary{calc,decorations.pathmorphing}
\pgfdeclarelayer{background}
\pgfdeclarelayer{foreground}
\pgfdeclarelayer{front}
\pgfsetlayers{background,main,foreground,front}

\let\epsilon=\varepsilon
\let\eps=\epsilon
\let\rho=\varrho
\let\theta=\vartheta

\let\phi=\varphi

\def\NN{{\mathds N}}

\newcommand{\cA}{\mathcal{A}}

\newcommand{\cK}{\mathcal{K}}

\newcommand{\cP}{\mathcal{P}}

\newcommand{\ccQ}{\mathscr{Q}}

\theoremstyle{plain}
\newtheorem{thm}{Theorem}[section]

\newtheorem{cor}[thm]{Corollary}
\newtheorem{lemma}[thm]{Lemma}

\theoremstyle{definition}

\newtheorem{dfn}[thm]{Definition}
\newtheorem{exmp}[thm]{Example}
\newtheorem{conj}[thm]{Conjecture}

\usepackage{accents}

\def\bl{\bigl(}
\def\br{\bigr)}

\DeclareMathOperator{\ex}{ex}

\usepackage{scalerel}

\newsavebox\vdegbox
\savebox\vdegbox{\tikz{
		\draw[black,fill=black] (90:1) circle (.35);
		\draw[black,line width=0.10cm] (210:1) circle (.30);
		\draw[black,line width=0.10cm] (330:1) circle (.30);
		\draw[opacity=0] (0:1.2) circle (0.1);
	}}

\newsavebox\vvbox
\savebox\vvbox{\tikz{
		\draw[black,line width=0.10cm] (90:1) circle (.30);
		\draw[black,fill=black] (210:1) circle (.35);
		\draw[black,fill=black] (330:1) circle (.35);
		\draw[opacity=0] (0:1.2) circle (0.1);
	}}

\newsavebox\pdegbox
\savebox\pdegbox{\tikz{
		\draw[black,line width=0.10cm] (90:1) circle (.30);
		\draw[black,fill=black] (210:1) circle (.35);
		\draw[black,fill=black] (330:1) circle (.35);
		\draw[black,line width=0.28cm ] (210:1) -- (330:1);
		\draw[opacity=0] (0:1.2) circle (0.1);
	}}

\newsavebox\vvvbox
\savebox\vvvbox{\tikz{
		\draw[black,fill=black] (90:1) circle (.35);
		\draw[black,fill=black] (210:1) circle (.35);
		\draw[black,fill=black] (330:1) circle (.35);
		\draw[opacity=0] (0:1.2) circle (0.1);
	}}
\newcommand{\vvv}{\mathord{\scaleobj{1.2}{\scalerel*{\usebox{\vvvbox}}{x}}}}
\newcommand{\pivvv}{\pi_{\vvv}}

\newsavebox\evbox
\savebox\evbox{\tikz{
		\draw[black,fill=black] (90:1) circle (.35);
		\draw[black,fill=black] (210:1) circle (.35);
		\draw[black,fill=black] (330:1) circle (.35);
		\draw[black,line width=0.28cm ] (210:1) -- (330:1);
		\draw[opacity=0] (0:1.2) circle (0.1);
	}}
\newcommand{\ev}{\mathord{\scaleobj{1.2}{\scalerel*{\usebox{\evbox}}{x}}}}	
\newcommand{\piev}{\pi_{\ev}}

\newsavebox\eebox
\savebox\eebox{\tikz{
		\draw[black,fill=black] (90:1) circle (.35);
		\draw[black,fill=black] (210:1) circle (.35);
		\draw[black,fill=black] (330:1) circle (.35);
		\draw[black,line width=0.28cm ] (90:1) -- (330:1);
		\draw[black,line width=0.28cm ] (90:1) -- (210:1);
		\draw[opacity=0] (0:1.2) circle (0.1);
	}}
\newcommand{\ee}{\mathord{\scaleobj{1.2}{\scalerel*{\usebox{\eebox}}{x}}}}
\newcommand{\piee}{\pi_{\ee}}

\newsavebox\eeebox
\savebox\eeebox{\tikz{
		\draw[black,fill=black] (90:1) circle (.35);
		\draw[black,fill=black] (210:1) circle (.35);
		\draw[black,fill=black] (330:1) circle (.35);
		\draw[black,line width=0.28cm ] (90:1) -- (330:1);
		\draw[black,line width=0.28cm ] (90:1) -- (210:1);
		\draw[black,line width=0.28cm ] (210:1) -- (330:1);
		\draw[opacity=0] (0:1.2) circle (0.1);
	}}

\newcommand{\Qvvv}{\ccQ^{(3)}_{\vvv}}
\newcommand{\Qev}{\ccQ^{(3)}_{\ev}}
\newcommand{\Qee}{\ccQ^{(3)}_{\ee}}

\begin{document}
\title[On a Tur\'an problem in weakly quasirandom $3$-uniform hypergraphs]{On a Tur\'an problem in weakly quasirandom $3$-uniform hypergraphs}

\author[Christian Reiher]{Christian Reiher}
\address{Fachbereich Mathematik, Universit\"at Hamburg, Hamburg, Germany}
\email{Christian.Reiher@uni-hamburg.de}

\author[Vojt\v{e}ch R\"{o}dl]{Vojt\v{e}ch R\"{o}dl}
\address{Department of Mathematics and Computer Science, 
Emory University, Atlanta, USA}
\thanks{The second author was supported by NSF grants DMS 1301698 and 1102086.}
\email{rodl@mathcs.emory.edu}

\author[Mathias Schacht]{Mathias Schacht}
\address{Fachbereich Mathematik, Universit\"at Hamburg, Hamburg, Germany}
\thanks{The third author was supported through the Heisenberg-Programme of the DFG\@.}
\email{schacht@math.uni-hamburg.de}


\begin{abstract}
Extremal problems for $3$-uniform hypergraphs  are known to be very difficult and despite 
considerable effort the progress has been slow. We suggest a more systematic study of extremal 
problems in the context of quasirandom hypergraphs. We say 
that a $3$-uniform hypergraph $H=(V,E)$ is \emph{weakly $(d,\eta)$-quasirandom} if for any subset 
$U\subseteq V$ the number of hyperedges of $H$ contained in $U$ is in the
interval $d\binom{|U|}{3}\pm\eta|V|^3$. We show that for any $\eps>0$ there exists  $\eta>0$ such that 
every sufficiently large 
weakly $(1/4+\eps,\eta)$-quasirandom hypergraph contains four vertices spanning at least three hyperedges.
This was conjectured by Erd\H os and S\'os and it is known that the density $1/4$ is best possible.

Recently, a computer assisted proof of this result based on the flag-algebra method was 
established by Glebov, Kr{\'a}\v{l}, and Volec.
In contrast to their work our proof presented here is based on the regularity method of hypergraphs and requires no 
heavy computations. In addition we obtain an ordered version of this result.
The method of our proof allows us to study extremal problems of this type 
in a more systematic way and we discuss a few extensions and open problems here. 
\end{abstract} 

\maketitle

%\tableofcontents

\section{Introduction}  
\subsection{Extremal problems for graphs and hypergraphs}
Given a fixed graph $F$ a typical problem in extremal graph theory asks for the maximum number of edges that a (large) graph~$G$ on $n$ vertices containing no copy of $F$ can have. More formally, for a fixed graph~$F$ let the \emph{extremal number $\ex(n, F)$} be the number $|E|$ 
of edges of an  $F$-free graph~$G=(V,E)$ on $|V|=n$ vertices with the maximum number of edges. It is well known and not hard to observe that the sequence 
$
%\frac{\mathrm{ex}(n, F)}{\binom{n}{2}}
\ex(n, F)/\binom{n}{2}
$
is decreasing. Consequently one may define the \emph{Tur\'an density}
\[
\pi(F)=\lim_{n\to\infty}  \frac{\ex(n, F)}{\binom{n}{2}}
\]
which describes the maximum density of large $F$-free graphs. The systematic study of these extremal parameters was initiated 
by Tur\'an~\cite{Tu41}, who determined $\ex(n,K_k)$ for complete graphs $K_k$. Recalling that the chromatic number $\chi(F)$ 
of a graph $F$ is the minimum number of colours one can assign to the vertices of $F$ in such a way that any two vertices connected by an edge receive distinct colours, it follows from a result of Erd\H{o}s and Stone~\cite{ErSt46} that
\[
\pi(F)=1-\frac1{\chi(F)-1}
\] 
(see also~\cite{ErSi66}, where the result in this form appeared first). 
In particular, the value of $\pi(F)$ can be calculated in finite time. It also follows that  the set 
$\{\pi(F)\colon F \text{ is a graph}\}$
of all Tur\'an densities of graphs 
%\[
%\bigl\{\pi(F)\colon F \text{ is a graph}\bigr\}
%\]
is given by
\[
\bigl\{0, \tfrac12, \tfrac23, \ldots,\tfrac{k-1}{k},\ldots\bigr\}\,.
\]

Already in his original work~\cite{Tu41} Tur\'an asked for hypergraph extensions of these extremal problems.
We restrict ourselves to \emph{$3$-uniform hypergraphs $H=(V,E)$}, where $V=V(H)$ is a finite set of \emph{vertices}
and the set of \emph{hyperedges} $E=E(H)\subseteq V^{(3)}$ is a family of the $3$-element subsets of the vertices.
Despite considerable effort, even for $3$-uniform hypergraphs $F$
no similar characterisation (as in the graph case) is known. Determining the value of $\pi(F)$ is a well known and hard problem even for ``simple'' hypergraphs like the complete $3$-uniform hypergraph $K_4^{(3)}$ on four vertices and $K_4^{(3)-}$, the hypergraph with four
vertices and three hyperedges. Currently the best known bounds for these Tur\'an densities are 
\[
	\frac{5}{9}\leq\pi(K_4^{(3)})\leq 0.5616 
	\qqand
	\frac{2}{7}\leq\pi(K_4^{(3)-})\leq 0.2871\,,
\] 
where the lower bounds are given by what is believed to be optimal constructions due to Tur\'an (see, e.g.,~\cite{Er77}) and Frankl and F\"uredi~\cite{FrFu84}.
The stated upper bounds are due to Razborov~\cite{Ra10} and Baber and Talbot~\cite{BaTa11}
and their proofs are based on the \emph{flag algebra method} introduced by Razborov~\cite{Ra07}. For a thorough discussion of Tur\'an type 
results and problems for hypergraphs we refer to the recent survey of Keevash~\cite{Ke11}.

\subsection{Quasirandom graphs and hypergraphs}
We consider a variant of Tur\'an type questions in connection with quasirandom hypergraphs. Roughly speaking,
a quasirandom hypergraph ``resembles'' a random hypergraph of the same edge density, by 
sharing some of the key properties with it, i.e., properties that hold true for the random  
hypergraph with probability close to $1$.

The 
investigation of quasirandom graphs was initiated with the observation
that several such properties of randomly generated graphs 
are equivalent in a deterministic sense. This phenomenon turned 
out to be useful and had a number of applications in combinatorics. 
The systematic study of quasirandom graphs was initiated  by Thomason~\cites{Th87a,Th87b} and by Chung, Graham, and Wilson~\cite{CGW89}. 
A pivotal feature of random graphs is the uniform edge distribution on ``large'' sets of vertices and a quantitative version
of this property is used to define quasirandom graphs. 

More precisely, a graph $G=(V, E)$ is \emph{quasirandom with density 
$d>0$} if for every subset of vertices $U\subseteq V$ the number $e(U)$ of edges contained in $U$ satisfies 
\begin{equation}\label{eq:qrG}
%\tag{$*$}
e(U)=d\tbinom{|U|}{2}+o(|V|^2)\,,
\end{equation}
where $o(|V|^2)/|V|^2\to0$ as $|V(G)|$ tends to infinity. Strictly speaking, we consider here 
a sequence of graphs $G_n=(V_n,E_n)$ where the number of vertices $|V_n|$ tends to infinity, but for the sake of 
a simpler presentation we will suppress the sequence in the discussion here. 
The main result in~\cite{CGW89} asserts, that satisfying~\eqref{eq:qrG} 
is deterministically equivalent to several other important properties of random graphs. In particular, it implies 
that for any fixed graph $F$ with $v_F$ vertices and $e_F$ edges
the number $N_F(G)$ of labeled copies of $F$ in a quasirandom graph $G=(V,E)$ of density $d$
satisfies
\begin{equation}\label{eq:qrN}
	N_F(G)=d^{e_F}|V|^{v_F}+o(|V|^{v_F})\,.
\end{equation}
In other words, the number of copies of $F$ %(injective) graph homomorphism $\psi\colon V(F)\longrightarrow V(G)$ 
is close to the expected value in a random graph with edge density $d$.

The analogous statement for hypergraphs fails to be true and uniform edge distribution on vertex sets 
is not sufficient to enforce a property similar
to~\eqref{eq:qrN} for all fixed $3$-uniform hypergraphs $F$ (see, e.g., Example~\ref{ex:tournament} below). 
A stronger notion of quasirandomness for which such an embedding result actually is true,
was considered in connection with the regularity lemma for hypergraphs (cf.\ Theorem~\ref{thm:RL} below).
% which allows to decompose an arbitrary hypergraphs into ``strongly quasirandom'' pieces -- a fact that can be used in a number of 
% applications.
The central notion for the work presented here, however,  is the straightforward 
extension of~\eqref{eq:qrG} to $3$-uniform hypergraphs, which was for example studied in~\cites{CHPS,KNRS}.

\begin{dfn}
\label{wqr}
A $3$-uniform hypergraph $H=(V, E)$ is \emph{weakly $(d, \eta)$-quasirandom}
if for every subset $U\subseteq V$ of vertices the number $e(U)$ of hyperedges contained in
$U$ satisfies
\begin{equation}\label{eq:defwqr}
	\big|e(U)-d\tbinom{|U|}{3}\big|\leq \eta\,n^3\,.
\end{equation}
\end{dfn}
%Weakly quasirandom hypergraphs were studied in~\cites{CHPS,KNRS}.
For future reference we note that a simple application of the sieve formula shows 
that the condition~\eqref{eq:defwqr}
implies 
\begin{equation}
\label{wqr-abc}
	\big|e(X, Y, Z)-d\,|X|\,|Y|\,|Z|\big|\leq 7\eta\, n^3
\end{equation}
for all $X, Y, Z\subseteq V$, where $e(X, Y, Z)$ denotes the number of triples $(x,y,z)\in X\times Y\times Z$
for which $\{x,y,z\}$ is a hyperedge of $H$. 
We shall denote by $\Qvvv(d,\eta)$ the class of all $3$-uniform weakly $(d,\eta)$-quasirandom hypergraphs, where
the three dots $\vvv$ appearing in the index of~$\ccQ$ 
symbolically represent the possible choices for the three sets $X$, $Y$, and~$Z$ 
from formula~$\eqref{wqr-abc}$. In fact, we will consider other classes of 
quasirandom $3$-uniform hypergraphs, which we will symbolise
by~$\Qev$ and~$\Qee$ and which we will investigate in connection with Tur\'an type question 
in~\cite{RRS-c} and~\cite{RRS-d} (see also Definition~\ref{12qr}).

\subsection{Extremal problems for weakly quasirandom hypergraphs}
Since in contrast to graphs, weakly quasirandom hypergraphs $H$ 
may not contain every fixed hypergraph $F$ it seems interesting to determine the 
maximum density~$d$ for which a weakly quasirandom $F$-free hypergraph of density $d$ exists.
This leads to the following notion of Tur\'an density for weakly quasirandom hypergraphs.
\begin{dfn} Given a $3$-uniform hypergraph $F$ we set
\begin{multline*}
	\pivvv(F)=\sup\bigl\{d\in[0,1]\colon \text{for every $\eta>0$ and $n\in \NN$ there exists an $F$-free,}\\
		\text{$3$-uniform hypergraph $H\in\Qvvv(d,\eta)$ with $|V(H)|\geq n$}\bigr\}\,.
\end{multline*}
\end{dfn}

Erd\H os and S\'os~\cite{ErSo82} (see also~\cite{Er90}) were the first to raise questions concerning $\pivvv(F)$. In particular, 
they suggested to study the cases when $F=K_4^{(3)-}$ or $F$ is a complete $3$-uniform hypergraph $K_k^{(3)}$. 
The following probabilistic construction, which can be traced back 
to the work of Erd\H os and Hajnal~\cite{ErHa72}, yields $\pivvv(K_4^{(3)-})\geq 1/4$.
\begin{exmp}
	\label{ex:tournament}
	Consider a random tournament $T_n$ on the vertex set $[n]=\{1,\dots,n\}$, i.e., an orientation of all edges 
	of the complete graph on the first $n$ positive integers such that each of the two directions $(i,j)$ or $(j,i)$
	of every pair of vertices $\{i,j\}$
	is chosen independently with probability $1/2$. Given such a tournament~$T_n$ we define 
	the $3$-uniform hypergraph~$H(T_n)$ on the same vertex set, by including the triple $\{i,j,k\}$ in $E(H(T_n))$
	if these three vertices span a cyclically oriented cycle of length three, 
	i.e., $\{i,j,k\}\in E(H(T_n))$ if either $(i,j)$, $(j,k)$, and~$(k,i)$ are all in $E(T_n)$ or  
	$(i,k)$, $(k,j)$, and $(j,i)$ are all in $E(T_n)$. It is easy to check that for every $\eta>0$
	with probability tending to $1$ as $n\to\infty$
	the hypergraph $H(T_n)$ is weakly $(1/4,\eta)$-quasirandom. Moreover, no hypergraph $H$
	obtained from a tournament in this way contains three hyperedges on four vertices, i.e., every such $H$
	is $K_4^{(3)-}$-free and this establishes $\pivvv(K_4^{(3)-})\geq 1/4$.
\end{exmp}
Recently, Glebov, Kr{\'a}{\v{l}}, and Volec~\cite{GKV} showed that the construction in Example~\ref{ex:tournament}
is optimal and proved
\[
	\pivvv(K_4^{(3)-})=\tfrac{1}{4}\,.
\]
The proof in~\cite{GKV} is computer assisted and based on the flag algebra method. We present
a computer free and very different proof of the same result. Moreover, our proof yields a strengthening of the result
which for ordered vertex sets guarantees the appearance of the~$K_4^{(3)-}$ in such a way that 
the apex vertex, that is the vertex  incident to three hyperedges of the $K_4^{(3)-}$, is either the first or the last.
Our method of proofs seems to open an approach to attack several other problems of this type and we shall discuss this 
in more detail in the concluding remarks in Section~\ref{sec:conc_rem}.       
\begin{thm}\label{K4-}
For every $\eps>0$ there exists an $\eta>0$ and an integer $n_0$ such that for every $n\geq n_0$
every 3-uniform  weakly $(\tfrac14+\eps, \eta)$-quasirandom hypergraph $H$
with vertex set~$V(H)=[n]$ contains a $K_4^{(3)-}$ in $H$ whose apex is either 
its smallest or its largest vertex. 
\end{thm}

Strictly speaking, the authors of~\cite{ErSo82} and~\cite{GKV}
considered a notion slightly different from the weak quasirandomness as defined in Definition~\ref{wqr}.
In their formulation they only required for an $n$-vertex hypergraph a lower bound of the form $e(U)\geq d\binom{|U|}{3}-\eta n^3$
for every set of vertices~$U$. However, a somewhat standard application of the so-called weak regularity lemma 
for hypergraphs (straightforward extension of Szemer\'edi's regularity lemma for graphs~\cite{Sz78}) 
implies that such a hypergraph contains a weakly $(d',\eta')$-quasirandom hypergraph on $cn$ vertices
for some $d'\geq d$, $c=c(d,\eta)>0$ and $\eta'$ with $\eta'\to0$ as~$\eta\to0$ and thus for the statement of 
Theorem~\ref{K4-} both assumption are equivalent (see, e.g.,~\cite{RRS-e}*{Proposition~2.5}).

\subsection*{Organisation} A central tool in the proof of Theorem~\ref{K4-} %and~\ref{zero}
is the regularity method for $3$-uniform hypergraphs
and we will introduce the relevant notation and results in Section~\ref{sec:regmethod}. Roughly speaking, the 
regularity lemma (Theorem~\ref{thm:RL}) allows us to decompose any given large hypergraph into quasirandom blocks. In fact, 
the blocks will enjoy stronger quasirandom properties (compared to Definition~\ref{eq:defwqr}), which in ``appropriate situations'' allow the embedding of 
any fixed hypergraph (see Theorem~\ref{thm:EL}). The main work in the proof of Theorem~\ref{K4-} %and~\ref{zero}
is to ensure such ``appropriate situations'' for embedding $K_4^{(3)-}$
after the application of the regularity lemma. These arguments will require 
several ideas from Ramsey theory and extremal graph theory. In particular, in the proof of Theorem~\ref{K4-} we will 
establish  a mean square degree condition in multipartite graphs that yields the existence of triangles and this result (Theorem~\ref{sync})
presented in the next section (Section~\ref{sec:forcing})
might be of independent interest. The proof of Theorem~\ref{K4-} will be given 
in Section~\ref{sec:K4-}. We close with a discussion of a few related results and 
open problems in Section~\ref{sec:conc_rem}.

\section{Forcing triangles in multipartite graphs}\label{sec:forcing}

In this section we shall prove a purely graph theoretic statement that will later be used in the proof of $\pivvv{(K_4^{(3)-})}=\tfrac14$. Essentially what we have to do then is to find a triangle in the link of a vertex of some weakly quasirandom $3$-uniform hypergraph~$H$, and after regularization this will become a problem about finding a triangle in some auxiliary multipartite graph. The vertices of this auxiliary graph will actually not correspond to the vertices of $H$ but rather to some bipartite graphs on $V(H)$, but this subtlety can be ignored until we reach Section~\ref{sec:K4-}.

The idea to study multipartite versions of, e.g., Mantel's theorem, or more generally of the 
Erd\H{o}s-Stone theorem, seems to go back at least to a suggestion by Bollob\'as, 
see the discussion after the proof of Theorem VI.2.15 in~\cite{BolEx}. The first systematic investigations of this kind have, to the best of our knowledge, been carried out by Bondy et al.\ in~\cite{BSTT}. In the case of triangles they showed the following: let $d_m$ denote the infimal real number with the property that any $m$-partite graph $G$ contains a triangle as soon as every edge density between two vertex classes of $G$ is greater than $d_m$. Then $d_m$ tends to $\tfrac12$ as $m$ tends to infinity, and moreover the statement $d_{\aleph_0}=\tfrac12$ about infinite-partite graphs with countably many vertex classes holds. In the other direction those authors 
showed that $d_4>1/2$. The situation was further clarified by Pfender~\cite{P} who proved that actually $d_m=\tfrac12$ holds for all $m\ge 12$
and determining the smallest $m$ with $d_m=1/2$ is an interesting open problem.

The theorem that follows is of a similar flavour.
%, even though implicationwise it is not comparable with the aforementioned results. 
We use the following notation. For an~$m$-partite graph $G=(V,E)$ with vertex classes $V_1\dcup\dots\dcup V_m=V$ we denote for every vertex $x\in V$
and $j\in[m]$ by $d_j(x)$ the size of the neighbourhood of $x$ in $V_j$.

\begin{thm}\label{sync}
For every $\eps>0$ there exists an integer $m$ such that if an $m$-partite graph~$G$
with nonempty vertex classes $V_1, \ldots, V_m$ satisfies
\[
\sum_{x\in V_i}d_j(x)^2\ge\bigl(\tfrac14+\eps\bigr)|V_i|\,|V_j|^2
\]
whenever $1\le i<j\le m$, then $G$ contains a triangle.
\end{thm}

\begin{proof}
For convenience we work with the hierarchy
\[
m^{-1}\ll m_*^{-1}\ll \delta\ll \eps\ll 1\,,
\]
and commence by defining a colouring of the pairs of indices from $[m]$ with integers from the interval $[1, (2\,\delta)^{-1}]$.

Let any $i$ and $j$ with $1\le i<j\le m$ be given. For each $r\in \NN$ we set
\[
Q_{ij}(r)=\bigl\{x\in V_i \colon d_j(x)\ge \bl\tfrac12+r\delta\br\,|V_j|\bigr\}\,.
\]
We contend that $|Q_{ij}(1)|\ge \delta\,|V_i|$. To see this, we split the right hand side of our assumption into two parts according to whether $x$ belongs to the set $Q=Q_{ij}(1)$ or not. This gives
\[
\bigl(\tfrac14+\eps\bigr)|V_i|\,|V_j|^2\le \bl\tfrac12+\delta\br^2|V_i-Q|\,|V_j|^2+|Q|\,|V_j|^2\,.
\]
Dividing by $|V_j|^2$ and using the trivial estimate $|V_i-Q|\le |V_i|$ we deduce
\[
\bigl(\tfrac14+\eps\bigr)|V_i|\le \bl\tfrac12+\delta\br^2|V_i|+|Q|\,,
\]
and by $\delta\ll\eps$ the desired conclusion follows.

Clearly, the set $Q_{ij}(r)$ becomes the smaller the larger we make if $r$, and if $r>(2\,\delta)^{-1}$ then $Q_{ij}(r)=\varnothing$ holds vacuously. Thus there exists a largest positive value of $r$, denoted by $r(i,j)$ in the sequel, for which $|Q_{ij}(r)|\ge \delta\,|V_i|$ 
holds. This concludes the definition of our colouring
\[
r\colon [m]^{(2)}\rightarrow \bigl\{1, 2, \ldots, \lfloor (2\,\delta)^{-1} \rfloor \bigr\}\,.
\]

By Ramsey's theorem, i.e., since we may assume the validity of the partition relation
\[
m\longrightarrow (m_*)^2_{\lfloor (2\,\delta)^{-1} \rfloor}\,,
\]
it is allowed to assume that after some relabeling there is a colour $r_*$ such that $r(i, j)=r_*$ holds whenever $1\le i<j\le m_*$. Of course, we should now find a triangle in $G$ with vertices from $V_1\cup \ldots \cup V_{m_*}$. It will turn out that there actually is such a triangle possessing a vertex in $V_1$.

Next we will single out some vertex from $V_1$ that will later be shown to appear in some triangle of $G$. To this end, we recall that $|Q_{1i}(r_*)|\ge \delta\,|V_1|$ holds for all $i\in\{2, \ldots, {m_*}\}$. It follows that the subsets $Q_{12}(r_*), \ldots, Q_{1m_*}(r_*)$ of $V_1$ cannot be disjoint provided we have chosen $m_*$ large enough. This means that some vertex $x\in V_1$ appears in at least two of them. For notational simplicity we assume $x\in Q_{12}(r_*)$ as well as $x\in Q_{13}(r_*)$ and endeavor to construct a triangle with vertices from $\{x\}\cup V_2\cup V_3$.

Let $A_2$ denote the set of neighbours of $x$ in $V_2$, set $B_2=V_2-A_2$, and define $A_3$ as well as $B_3$ analogously. The choice of $x$ implies 
$|A_2|=d_2(x)\ge \bl\tfrac12+r_*\delta\br|V_2|$ 
and 
$A_3(x)\ge \bl\tfrac12+r_*\delta\br|V_3|$.  
Any edge between $A_2$ and $A_3$ gives rise to a triangle of the desired kind, so for the sake of a contradiction we will henceforth assume that no such edge would exist.

Then we have $d_3(y)\le |B_3|\le \bl\tfrac12-r_*\,\delta\br|V_3|$ for all $y\in A_2$. For $y\in B_2$ we either have $d_3(y)<\bl\tfrac12+(r_*+1)\delta\br|V_3|$, or $y$ belongs to the set $C=Q_{23}(r_*+1)$. But the maximality of $r_*=r(2,3)$ implies 
$|C|<\delta\,|V_2|$ and for $y\in C$ we still have $d_3(y)\le |V_3|$. So dividing the right hand side of the assumption 
\[
\bigl(\tfrac14+\eps\bigr)|V_2|\,|V_3|^2 \le \sum_{y\in V_2}d_3(y)^2
\]
into three parts depending on whether $y$ appears in $A_2$, $B_2-C$ or $C$ we derive
\[
\bigl(\tfrac14+\eps\bigr)|V_2|\,|V_3|^2 \le 
|A_2|\bl\tfrac12-r_*\,\delta\br^2|V_3|^2
+|B_2|\bl\tfrac12+(r_*+1)\delta\br^2|V_3|^2
+\delta\,|V_2|\,|V_3|^2\,.  
\]
Since $|A_2|\ge \bl\tfrac12+r_*\delta\br|V_2|>\tfrac12\,|V_2|$ and $|A_2|+|B_2|=|V_2|$, this implies
\[
\tfrac14+\eps\le 
\bl\tfrac12+r_*\delta\br\bl\tfrac12-r_*\,\delta\br^2
+\bl\tfrac12-r_*\delta\br\bl\tfrac12+(r_*+1)\delta\br^2+\delta\,.
\]
Now $\tfrac12+r_*\delta\le 1$ and for each $x\in [0,1]$ we have $(x+\delta)^2\le x^2+3\delta$ by $\delta\le 1$, so %it follows that
\[
\tfrac14+\eps\le 
\bl\tfrac12+r_*\delta\br\bl\tfrac12-r_*\,\delta\br^2
+\bl\tfrac12-r_*\delta\br\bl\tfrac12+r_*\delta\br^2+4\,\delta\,.
\]
Here, the sum of the first two terms gives $\tfrac14-(r_*\,\delta)^2$ and hence at most $\tfrac14$, so that altogether we get $\eps\le 4\,\delta$, contrary to $\delta\ll\eps$. Thereby Theorem~\ref{sync} is proved.
\end{proof}

%\begin{rem}\label{rem:Fsync}
The authors of the articles cited at the beginning of this section actually studied the more general question of finding larger cliques, or even arbitrary graphs, in dense multipartite graphs, obtaining results comparable to those indicated above. Similarly, the proof of Theorem~\ref{sync}
generalises in a straightforward way from triangles to arbitrary cliques $K_k$ and we omit the details.
\begin{thm}\label{thm:Fsync}
For every $\eps>0$ and $k\geq 3$ there exists an integer $m$ such that if an~$m$-partite graph~$G$
with nonempty vertex classes $V_1, \ldots, V_m$ satisfies
\[
\sum_{x\in V_i}d_j(x)^2\ge\Bigl(\big(\tfrac{k-2}{k-1}\big)^2+\eps\Bigr)|V_i|\,|V_j|^2
\]
whenever $1\le i<j\le m$, then $G$ contains a clique $K_k$.\qed
\end{thm}
In fact, the proof guarantees~$\Omega(n^k)$ copies of~$K_k$ and as a result we may 
replace $K_k$ in  Theorem~\ref{thm:Fsync} by an arbitrary graph~$F$ with chromatic number $\chi(F)=k$.
%\end{rem}

\section{Hypergraph regularity method}\label{sec:regmethod}
A key tool in the proof of Theorem~\ref{K4-} %and~\ref{zero} 
is the regularity lemma for $3$-uniform hypergraphs. 
We follow the approach from~\cites{RoSchRL,RoSchCL} combined with the results from~\cite{Gow06} and~\cite{NPRS09}
and below we introduce the necessary notation.

For two disjoint sets $X$ and $Y$ we denote by $K(X,Y)$ the complete bipartite graph with that vertex partition.
We say a bipartite graph $P=(X\dcup Y,E)$ ist \emph{$(\delta_2, d_2)$-regular} if for all subsets 
$X'\subseteq X$ and $Y'\subseteq Y$ we have 
\[
	\big|e(X',Y')-d_2|X'||Y'|\big|\leq \delta_2 |X||Y|\,,
\]
where $e(X',Y')$ denotes the number of edges of $P$ with one vertex in $X'$ and one vertex in~$Y'$.
Moreover, for $k\geq 2$ we say a $k$-partite graph $P=(X_1\dcup \dots\dcup X_k,E)$ is $(\delta_2, d_2)$-regular, 
if all its  $\binom{k}{2}$ naturally 
induced bipartite subgraphs $P[X_i,X_j]$ are $(\delta_2, d_2)$-regular. 
For a tripartite graph $P=(X\dcup Y\dcup Z,E)$
we denote by $\cK_3(P)$ the triples of vertices spanning a triangle in~$P$, i.e., 
\[
	\cK_3(P)=\{\{x,y,z\}\subseteq X\cup Y\cup Z\colon xy, xz, yz\in E\}\,.
\]
If the tripartite graph $P$ is $(\delta_2, d_2)$-regular, then the so-called \emph{triangle counting lemma}
implies that 
\begin{equation}
	\label{eq:TCL}
		|\cK_3(P)|\leq d_2^3|X||Y||Z|+3\delta_2|X||Y||Z|\,.
\end{equation}

We say a $3$-uniform hypergraph $H=(V,E_H)$ is regular w.r.t.\ a tripartite graph $P$ if it matches 
approximately
the same proportion of triangles for every subgraph $Q\subseteq P$. This we make precise in the following definition.

\begin{dfn}
\label{def:reg}
A $3$-uniform hypergraph $H=(V,E_H)$ is \emph{$(\delta_3,d_3)$-regular w.r.t.\ 
a tripartite graph $P=(X\dcup Y\dcup Z,E_P)$} 
with $V\supseteq  X\cup Y\cup Z$ if for every tripartite subgraph $Q\subseteq P$ we have 
\[
	\big||E_H\cap\cK_3(Q)|-d_3|\cK_3(Q)|\big|\leq \delta_3|\cK_3(P)|\,.
\]
Moreover, we simply say \emph{$H$ is $\delta_3$-regular w.r.t.\ $P$}, if it is $(\delta_3,d_3)$-regular for some $d_3\geq 0$.
We also define the \emph{relative density} of $H$ w.r.t.\ $P$ 
\[
	d(H\,|\,P)=\frac{|E_H\cap\cK_3(P)|}{|\cK_3(P)|}\,,
\]
where we use the convention $d(H\,|\,P)=0$ if $\cK_3(P)=\emptyset$.
\end{dfn}

The regularity lemma for $3$-uniform hypergraphs, introduced by Frankl and R\"odl in~\cite{FR}, provides for 
every hypergraph $H$ a partition of its vertex set and a partition of the edge sets of the complete bipartite 
graphs induced by the vertex partition such that for appropriate constants $\delta_3$, $\delta_2$ and $d_2$ 
\begin{enumerate}[label=\nlabel]
	\item the bipartite graphs given by the partitions are $(\delta_2,d_2)$-regular and
	\item $H$ is $\delta_3$-regular for ``most'' tripartite graphs given by the partition.
\end{enumerate}
Here we use a refined version from~\cite{RoSchRL}*{Theorem~2.3}.
\begin{thm}[Regularity Lemma]
	\label{thm:RL}
	For all $\delta_3>0$, $\delta_2\colon \NN \to (0,1]$, and~$t_0\in\NN$ 
	there exists an integer $T_0$ such that for every $n\geq t_0$
	and every $n$-vertex $3$-uniform hypergraph $H=(V,E_H)$ the following holds.
	
	There are integers $t$ and $\l$ with $t_0\leq t\leq T_0$ and $\l\leq T_0$
	and there exists a vertex partition $V_0\dcup V_1\dcup\dots\dcup V_t=V$ 
	and for all $1\leq i<j\leq t$ there exists 
	a partition 
	\[
		\cP^{ij}=\{P^{ij}_\alpha=(V_i\dcup V_j,E^{ij}_\alpha)\colon 1\leq \alpha \leq \l\}
	\] 
	of the edge set of the complete bipartite graph $K(V_i,V_j)$ satisfying the following properties
	\begin{enumerate}[label=\rmlabel]
		\item $|V_0|\leq \delta_3 n$ and $|V_1|=\dots=|V_t|$,
		\item for all $1\leq i<j\leq t$ and $\alpha\in [\l]$ the bipartite graph $P^{ij}_\alpha$ is $(\delta_2(\l),1/\l)$-regular, and
		\item\label{RL:3} $H$ is $\delta_3$-regular w.r.t.\ $P^{ijk}_{\alpha\beta\gamma}$
			for all but at most $\delta_3t^3\l^3$ tripartite graphs 
			\begin{equation}\label{eq:triad}
				P^{ijk}_{\alpha\beta\gamma}=P^{ij}_\alpha\dcup P^{ik}_\beta\dcup P^{jk}_\gamma=(V_i\dcup V_j\dcup V_k, E^{ij}_\alpha\dcup E^{ik}_{\beta}\dcup E^{jk}_{\gamma})\,,
			\end{equation}
			with $1\leq i<j<k\leq t$ and $\alpha$, $\beta$, $\gamma\in[\l]$.
	\end{enumerate}
\end{thm}
Owing to their special r\^ole we shall refer to the tripartite graphs considered in~\eqref{eq:triad} as \emph{triads}.
In the formulation of the regularity lemma in~\cite{RoSchRL} a more refined version of hypergraph regularity was used. 
However, owing to the results from~\cite{Gow06} and~\cite{NPRS09}*{Corollaries~2.1 and~2.3} for our purposes here the version 
from Definition~\ref{def:reg} suffices. 

Similarly as in other proofs based on the regularity method it will
be convenient to ``clean'' the regular partition provided by Theorem~\ref{thm:RL}. In particular, 
we shall disregard hyperedges of $H$ that ``belong'' to \emph{irregular} or \emph{sparse} triads of the regular partition.
Since by property~\ref{RL:3} \emph{globally} 
$H$ is not regular for up to at most $\delta_3t^3\l^3$ triads, 
a simple averaging argument shows that 
for $m=m(\delta_3)$  (with $m\to\infty$ as $\delta_3\to0$) there exist 
vertex classes $V_{i_1},\dots,V_{i_m}$ such that for all fixed $1\leq a<b<c\leq m$
\emph{locally} $H$ is $\delta_3$-regular 
for all but at most $\sqrt{\delta_3}\l^3$ triads $P^{i_ai_bi_c}_{\alpha\beta\gamma}$ with $\alpha$, $\beta$, $\gamma\in[\l]$.
After removal of the hyperedges belonging to irregular or sparse triads, these considerations 
lead to the following immediate consequence of Theorem~\ref{thm:RL}.

\begin{cor}
	\label{cor:TuRL}
	For every $d_3>0$, $\delta_3>0$ and  $m\in\NN$, and every function $\delta_2\colon \NN \to (0,1]$,
	there exist integers~$T_0$ and $n_0$ such that for every $n\geq n_0$
	and every $n$-vertex $3$-uniform hypergraph $H=(V,E)$ the following holds.
	
	There exists a subhypergraph $\hat H=(\hat V,\hat E)\subseteq H$, a positive integer $\l\leq T_0$,
	a vertex partition $V_1\dcup\dots\dcup V_m=\hat V$, 
	and for all $1\leq i<j\leq m$ there exists 
	a partition of pairs $\cP^{ij}=\{P^{ij}_\alpha=(V_i\dcup V_j,E^{ij}_\alpha)\colon 1\leq \alpha \leq \l\}$ 
	of $K(V_i,V_j)$ satisfying the following properties
	\begin{enumerate}[label=\rmlabel]
		\item $|V_1|=\dots=|V_m|\geq (1-\delta_3)n/T_0$,
		\item for every $1\leq i<j\leq m$ and $\alpha\in [\l]$ the bipartite graph $P^{ij}_\alpha$ is $(\delta_2(\l),1/\l)$-regular,
		\item $\hat H$ is $\delta_3$-regular w.r.t.\ $P^{ijk}_{\alpha\beta\gamma}$
			for \emph{all} tripartite graphs $P^{ijk}_{\alpha\beta\gamma}$ 
			with $1\leq i<j<k\leq m$ and $\alpha$, $\beta$, $\gamma\in[\l]$, where either $d(\hat H\,|\,P)=0$ or
			$d(\hat H\,|\,P)\geq d_3$, and
		\item\label{TuRL:4} for every $1\leq i<j<k\leq m$ we have 
			\[
				e_{\hat H}(V_i,V_j,V_k)\geq e_{H}(V_i,V_j,V_k)-(d_3+\delta_3)|V_i||V_j||V_k|\,.
			\]
	\end{enumerate}
	Moreover, if the vertex set $V=[n]$ then we can ensure $\max(V_i)<\min(V_{i+1})$
	for every $i=1,\dots,m-1$.
\end{cor}
\begin{proof}
	For the proof of Corollary~\ref{cor:TuRL} (including the moreover-part)
	we shall apply the regularity lemma (Theorem~\ref{thm:RL})
	with $\delta_3'$ sufficiently small  such that 
	\begin{equation}\label{eq:pTuRL}
		\delta_3'<\delta_3^2 \qand \Big(1-11\sqrt{\delta_3'}\Big)\binom{m}{3} > \binom{m}{3}-1
	\end{equation}
	and with the integer $t_0=\max(m,\lceil1/\delta'_3\rceil,21)$ and the given function $\delta_2\colon \NN\to(0,1]$.

	We recall 
	that the hypergraph regularity lemma is proved by iterated refinements
	starting with an arbitrary initial partition. Hence, given the hypergraph $H=(V,E)$ with $V=[n]$ we may split~$V$ 
	initially into $t_0$ equal sized intervals $I_1\dcup\dots\dcup I_{t_0}=[n]$ and 
	then the vertex partition $V_1\dcup\dots\dcup V_t$ provided by the regularity lemma, Theorem~\ref{thm:RL},
	will refine this initial partition of intervals. 
	
	We consider an auxiliary $3$-uniform hypergraph $R=([t],E_R)$ on the vertex set $[t]$, where a hyperedge $\{i,j,k\}$ 
	signifies the following two properties:
	\begin{enumerate}[label=\alabel]
	\item\label{it:pTuRLi} at most $\sqrt{\delta_3'}\l^3$ triads $P^{ijk}_{\alpha\beta\gamma}$ are not $\delta_3'$-regular and
	\item\label{it:pTuRLii} $V_i$, $V_j$, and $V_k$ are contained in three different initial intervals $I_s$. 
	\end{enumerate}
	Property~\ref{RL:3} of Theorem~\ref{thm:RL} and $t\geq t_0\geq 21$
	asserts that at most 
	\[
		\frac{\delta'_3t^3\l^3}{\sqrt{\delta_3'}\l^3}
		=
		\sqrt{\delta'_3}t^3
		<
		7\sqrt{\delta'_3}\binom{t}{3}		
	\]
	triples $\{i,j,k\}$ fail to satisfy~\ref{it:pTuRLi}.
	Moreover, at most 
	\[
		t_0\cdot\binom{t/t_0}{2}t<\frac{4}{t_0}\binom{t}{3}
	\]
	triples are excluded because of~\ref{it:pTuRLii}. Consequently, owing to the choice of $t_0\geq 1/\delta'_3$ 
	we infer that auxiliary hypergraph $R$ has density at least $(1-11\sqrt{\delta'_3})$.
	The choice of $\delta'_3$ in~\eqref{eq:pTuRL} entails that $R$ contains a 
	clique on $m$ vertices, say, $i_1,\dots,i_m$. Again appealing to~\ref{it:pTuRLii}
	of the construction of $R$ we may assume that there are indices $1\leq j_1<\dots<j_m\leq t_0$
	such that $V_{i_k}\subseteq I_{j_k}$ for every $k\in [m]$ and, consequently, these vertex sets
	satisfy the moreover-part of the Corollary~\ref{cor:TuRL}.

	In order to construct the desired hypergraph $\hat H$ we remove hyperedges of $H$ that are contained in triads with density less than
	$d_3$, i.e., hyperedges $e\in E_H\cap\cK_3(P^{ijk}_{\alpha\beta\gamma})$ when
	 $d(H\,|\,P^{ijk}_{\alpha\beta\gamma})<d_3$.
	 Moreover, we remove hyperedges of $H$ that are contained in triads~$P^{ijk}_{\alpha\beta\gamma}$ for which 
	 $H$ is not $\delta_3$-regular and let $H_0$ be the hypergraph that remained after these deletions. 
	 Finally it follows that setting $\hat H$ to the subhypergraph of $H_0$ induced on $V_{i_1}\dcup\dots\dcup V_{i_m}$
	 has the desired properties.
\end{proof}

We shall use a so-called \emph{counting/embedding lemma}, which allows us to embed hypergraphs of fixed isomorphism type 
into appropriate and sufficiently regular and dense triads of the regular partition provided by the regularity lemma. 
The following statement is a direct consequence of~\cite{NPRS09}*{Corollary~2.3}.
\begin{thm}[Embedding Lemma]
	\label{thm:EL}
	For every $3$-uniform hypergraph $F=(V_F,E_F)$ with vertex set $V_F=[f]$
	and every $d_3>0$ there exists $\delta_3>0$, and functions 
	$\delta_2\colon \NN\to(0,1]$ and  $N\colon \NN\to\NN$
	such that the following holds for every $\l\in\NN$.
	
	Let $P\!=\!(\bigdcup_{i\in[f]}V_i, E_P)$ be a $(\delta_2(\l),\frac{1}{\l})$-regular, $f$-partite graph
	with $|V_1|\!=\!\dots\!=\!|V_f|\geq N(\l)$ and let $H$ be an $f$-partite, $3$-uniform hypergraph
	satisfying for every edge $ijk\in E_F$
	\begin{enumerate}[label=\alabel]
		\item\label{EL:a} $H$ is $\delta_3$-regular w.r.t.\ to the tripartite graph $P[V_i\dcup V_j\dcup V_k]$ and 
		\item $d(H\,|\,P[V_i\dcup V_j\dcup V_k])\geq d_3$
	\end{enumerate} 
	then $H$ contains a copy of $F$, where for every $i\in [f]=V_F$ the image of $i$ is contained in~$V_i$.
\end{thm}
In an application of Theorem~\ref{thm:EL} the tripartite graphs $P[V_i\dcup V_j\dcup V_k]$ in~\ref{EL:a}
will be given by triads  $P^{ijk}_{\alpha\beta\gamma}$ from the partition given by the regularity lemma.

We shall consider weakly quasirandom hypergraphs $H$ of density $\mu$ bounded away from $0$.
In particular, this assumption implies that in any regular partition provided by Theorem~\ref{thm:RL},
we have the property that the density of $H$ induced on any three vertex classes $V_i$, $V_j$, and~$V_k$ 
will be close to $\mu$. For fixed $i$, $j$ and $k$ 
this only implies that $d(H\,|\,P^{ijk}_{\alpha\beta\gamma})\sim\mu$ on the average taken over all 
$\l^3$ choices of $\alpha$, $\beta$, and $\gamma\in[\l]$. 
This, however, gives only little information on the density of $H$ w.r.t.\ a particular $P^{ijk}_{\alpha\beta\gamma}$.
Consequently, for the proof of Theorem~\ref{K4-} further analysis is required to arrive at a situation ready for an application of  
Theorem~\ref{thm:EL}. This will be the focus in Section~\ref{sec:K4-}.
% and~\ref{sec:jump} 
%for Theorems~\ref{K4-} and~\ref{zero}, respectively.

\section{Embedding \texorpdfstring{$K_4^{(3)-}$}{K4-}}
\label{sec:K4-}

In this section we deduce Theorem~\ref{K4-}. The proof will be based on the regularity lemma for hypergraphs in form of 
Corollary~\ref{cor:TuRL} and the embedding lemma (Theorem~\ref{thm:EL}). 
Below we reduce the proof of Theorem~\ref{K4-} to a lemma (see Lemma~\ref{K4-modif} below)
which locates in a sufficiently regular partition of a weakly quasirandom hypergraph with density 
$>1/4$ a collection of triads that are ready for an application of the embedding lemma for~$K_4^{(3)-}$.

\begin{proof}[Proof of Theorem~\ref{K4-}]
Given $\eps>0$ we have to find appropriate $\eta>0$ and $n_0\in\NN$. For this purpose we start by choosing some 
auxiliary constants obeying the hierarchy 
\[
\delta_3\ll d_3, m^{-1}\ll\eps\,.
\]
For these choices of $\delta_3$ and $d_3$ and $F=K_4^{(3)-}$ we appeal to Theorem~\ref{thm:EL} and 
obtain $\delta_2\colon\NN\to\NN$ and $N\colon \NN\to\NN$. Without loss of generality we may assume 
that for all $\ell\in\NN$ we have
\[
\delta_2(\ell)\ll\ell^{-1}, \eps\,.
\]
Applying Corollary~\ref{cor:TuRL} to $d_3$, $\delta_3$, $m$, and $\delta_2$ we 
get two integers $n'_0$ and $T_0$. Now we claim that any 
\[
\eta\ll T_0^{-1} \qquad \text{ and } \qquad n_0\gg n'_0\,, T_0\cdot N(T_0)
\]
are as desired.

To justify this, we let any weakly $(1/4+\eps, \eta)$-quasirandom hypergraph $H$ on $n\ge n_0$ vertices be given. Since $n\ge n'_0$ holds as well, we may apply Corollary~\ref{cor:TuRL}, thus getting a subhypergraph $\hat H\subseteq H$ with vertex partition $\hat V=V_1\dcup\dots\dcup V_m$ and edge partitions
$\cP^{ij}=\{P^{ij}_{\alpha}\colon \alpha\in[\l]\}$ of $K(V_i,V_j)$ for $1\le i<j\le m$.

In view of the embedding lemma (Theorem~\ref{thm:EL}) the task that remains to be done is now reduced to the task of locating four vertex classes
$V_{i_1},\dots,V_{i_4}$ with $i_1< i_2< i_3<i_4$, $\max(V_{i_a})<\min(V_{i_{a+1}})$ for $a=1, 2, 3$, and six bipartite graphs $P^{ab}\in \cP^{i_ai_b}$ for $1\le a<b\le4$
from the regular partition, such that at least three of the $\binom{4}{3}$ triads 
\[
	P^{abc}=P^{ab}\dcup P^{ac}\dcup P^{bc}
\] 
with $1\le a<b<c\le4$
are \emph{dense} and \emph{regular}, i.e., $d(H\,|\,P^{abc})\geq d_3$ and $H$ is $\delta_3$-regular w.r.t.~$P^{abc}$. 
For the moreover-part, we also have to make sure that we embed the apex vertex of $K_4^{(3)-}$ either into $V_{i_1}$ or into $V_{i_4}$. This will be rendered by Lemma~\ref{K4-modif} (stated below).

In fact due to property~\ref{TuRL:4} of Corollary~\ref{cor:TuRL}
and the weak quasirandomness of $H$ given by the assumption of Theorem~\ref{K4-} (see~\eqref{wqr-abc}) we have 
\begin{align}
		e_{\hat H}(V_i,V_j,V_k) &\geq \left(\frac{1}{4}+\eps\right)|V_i||V_j||V_k|-(d_3+\delta_3)|V_i||V_j||V_k|-7\eta n^3 \nonumber\\
			&\geq \left(\frac{1}{4}+\frac{\eps}{2}\right)|V_i||V_j||V_k|\,,\label{eq:hHqr}
\end{align}
where the last step exploits $d_3, \delta_3\ll\eps$ and $\eta\ll T_0^{-1}$.

Moreover, since every triad $P^{ijk}_{\alpha\beta\gamma}$
is $(\delta_2(\l),1/\l)$-regular (as a tripartite graph), the triangle counting lemma for graphs (see~\eqref{eq:TCL})
asserts that it spans at most $(1/\ell^3+3\delta_2(\l))|V_i||V_j||V_k|$ triangles. By our choice of the function $\delta_2$
we deduce from~\eqref{eq:hHqr} 
that for every fixed $1\leq i<j<k\leq m$ at least 
\begin{equation}\label{eq:hHqrA}
	\frac{\left(1/4+\eps/2\right)}{(1+3\delta_2(\l)\l^3)}\l^3>\left(\frac{1}{4}+\frac{\eps}{4}\right)\l^3
\end{equation}
triads $P^{ijk}_{\alpha\beta\gamma}$ satisfy $d(H\,|\,P^{ijk}_{\alpha\beta\gamma})\geq d_3$.

For fixed $1\leq i<j<k\leq m$ we consider an auxiliary tripartite $3$-uniform hypergraph 
$\cA^{ijk}$ with vertices corresponding to bipartite graphs from the regular partition
and hyperedges representing dense triads. More precisely, we set 
$V(\cA^{ijk})=\cP^{ij}\dcup\cP^{ik}\dcup\cP^{jk}$
and we include the triple  $P^{ij}_\alpha P^{ik}_\beta P^{jk}_\gamma$ in  $E(\cA^{ijk})$ if
$d(\hat H\,|\,P^{ijk}_{\alpha\beta\gamma})\geq d_3$. 
This way~\eqref{eq:hHqrA}
translates to the assertion, that~$\cA^{ijk}$ contains at least $(1/4+\eps/4)\l^3$ hyperedges.

In Lemma~\ref{K4-modif} we analyse the $\binom{m}{2}$-partite, $3$-uniform
hypergraph~$\cA$ given by the union of all
$\cA^{ijk}$ with $1\leq i<j<k\leq m$.
%, i.e.,~$\cA$ is an $\binom{m}{2}$-partite, $3$-uniform hypergraph.
Note that only $\binom{m}{3}$ of the $\binom{\binom{m}{2}}{3}$ naturally induced  tripartite subhypergraphs 
of~$\cA$ span any hyperedges. Lemma~\ref{K4-modif} asserts that such a hypergraph~$\cA$ contains three hyperedges 
on six vertices, which translates back to four vertex classes
$V_{i_1},\dots,V_{i_4}$ and six bipartite graphs $P^{ab}\in \cP^{i_ai_b}$ for $1\le a<b\le4$
from the regular partition of $\hat H$, such that at least three of the four triads $P^{abc}$ with $1\le a<b<c\le4$
satisfy $d(\hat H\,|\,P^{abc})\geq d_3$. Since $\hat H$ was $\delta_3$-regular for any triad, this shows that 
the assumptions of the embedding lemma, Theorem~\ref{thm:EL}, are met for $F=K_4^{(3)-}$
%if we chose $\delta_2$ sufficiently small as a function of $\l$ and made $n_0$ so large that $\ell\le T_0$ and $|V_i|\ge (1-\delta_3)n/T_0$ imply $|V_i|>N(\ell)$ for $i\in [m]$. 
and, therefore, $\hat H\subseteq H$ 
contains a copy of $K_4^{(3)-}$. We also note that the moreover-part of Lemma~\ref{K4-modif} 
together with the moreover-part of Corollary~\ref{cor:TuRL} implies that there exists indeed a 
copy of $K_4^{(3)-}$ in $H=([n],E)$ with the  apex vertex either in the front or at the end.
This concludes the reduction of Theorem~\ref{K4-} to Lemma~\ref{K4-modif} (where $\eps$ corresponds
to $\eps/4$ in the reduction above).
%  Finally we mention that the $\eps$ in Lemma~\ref{K4-modif} corresponds to $\eps/4$ here.
\end{proof}

\begin{lemma}\label{K4-modif}
For every $\eps>0$ there exists an integer $m$ such  that the following holds.
If $\cA$ is an $\binom{m}{2}$-partite $3$-uniform hypergraph with 
\begin{enumerate}[label=\rmlabel]
\item nonempty vertex classes $\cP^{ij}$ for 
$1\le i<j\le m$ such that
\item for each triple $1\le i<j<k\le m$ the restriction $\cA^{ijk}$ of $\cA$ to $\cP^{ij}\cup \cP^{ik}\cup \cP^{jk}$ 
contains at least $\bl\tfrac 14+\eps\br |\cP^{ij}|\,|\cP^{ik}|\,|\cP^{jk}|$ 
triples,
\end{enumerate}
then there are four distinct indices $i_1$, $i_2$, $i_3$, and $i_4$ from~$[m]$ together 
with six vertices $P^{ab}\in \cP^{i_ai_b}$ for $1\le a<b\le4$ such that 
$P^{12}P^{14}P^{24}$, $P^{13}P^{14}P^{34}$, and~$P^{23}P^{24}P^{34}$ are triples of $\cA$. 

Moreover, there exists such a configuration with
\[
	i_4=\max(i_1, i_2, i_3, i_4) \qquad \text{or} \qquad i_4=\min(i_1, i_2, i_3, i_4)\,.
\]
\end{lemma}
%The moreover-part of Lemma~\ref{K4-modif} comes out automatically of the proof, but we will only need it
% for the strengthening of Theorem~\ref{K4-}
%disussed in Section~\ref{sec:conc_rem}.

\begin{proof} Suppose
\[
m^{-1}\ll m_*^{-1}\ll \eps\,,
\]
and let a $3$-uniform hypergraph $\cA$ as in Lemma~\ref{K4-modif} be given. Notice that each of the three vertices $P^{12}$, $P^{13}$, and $P^{23}$ appears only once in the conclusion, so we may eliminate them from consideration by ``projecting'' the nonempty 
tripartite parts of $\cA$ onto appropriate bipartite graphs. That is to say that for any three distinct indices $i$, $j$, and $k$ from $[m]$ we define a bipartite graph $Q^i_{jk}$ with bipartition $(\cP^{ij}, \cP^{ik})$ by putting an edge between 
$P^{ij}\in \cP^{ij}$ and $P^{ik}\in \cP^{ik}$ if and only if for some $P^{jk}\in \cP^{jk}$ the triple 
$P^{ij}P^{ik}P^{jk}$ belongs to~$E(\cA^{ijk})$.

In the next step of the argument, we colour the $3$-subsets of $[m]$ with two colours, called {\it red} and {\it green}, with the intention of applying Ramsey's Theorem afterwards. So let any three indices $1\le i<j<k\le m$ be given. Each triple 
$P^{ij}P^{ik}P^{jk}$ from $E(\cA^{ijk})$, with $P^{ij}\in \cP^{ij}$, $P^{ik}\in \cP^{ik}$, and $P^{jk}\in \cP^{jk}$, 
gives rise to a unique pair $(P^{ij}P^{ik}, P^{ik}P^{jk})$ of edges $P^{ij}P^{ik}\in E\bl Q^{i}_{jk}\br$ and 
$P^{ik}P^{jk}\in E\bl Q^k_{ij}\br$ and hence our assumption on the density of $\cA^{ijk}$ yields 
\[
\sum_{P^{ik}\in \cP^{ik}}d_{Q^{i}_{jk}}(P^{ik})d_{Q^{k}_{ij}}(P^{ik})\ge \bl\tfrac 14+\eps\br |\cP^{ij}|\,|\cP^{ik}|\,|\cP^{jk}|\,.  
\]  
Thus the Cauchy-Schwarz inequality informs us that at least one of the two estimates
\[
\tag{$*$}
\sum_{P^{ik}\in \cP^{ik}}d^{\phantom{!}2}_{Q^{i}_{jk}}(P^{ik})\ge \bl\tfrac 14+\eps\br |\cP^{ij}|^2\,|\cP^{ik}| 
\]
or
\[ 
\tag{$**$}
\sum_{P^{ik}\in \cP^{ik}}d^{\phantom{!}2}_{Q^{k}_{ij}}(P^{ik})\ge \bl\tfrac 14+\eps\br |\cP^{jk}|^2\,|\cP^{ik}| 
\]
holds. Hence there can arise no clash of colours if we resolve to colour $\{i, j, k\}$ red if $(*)$ fails and green if $(**)$ fails. If both $(*)$ and $(**)$ are valid, the colour of $\{i, j, k\}$ is irrelevant and we make and arbitrary choice. In other words, if $\{i, j, k\}$ ends up being red, then necessarily~$(**)$ holds, whilst if this triple is green, then this indicates the validity of $(*)$.

By Ramsey's Theorem, or more precisely as we may assume the partition relation 
\[
m\longrightarrow (m_*)^3_2\,,
\]
there is a set $X\subseteq [m]$ of size $m_*$ such that all triples from $X$ have the same colour. Due to symmetry it is allowed to assume that this common colour is red, and relabeling our indices if necessary we may further suppose that $X=[m_*]$. We contend that a configuration of the desired kind can be found with $1\le i_1<i_2<i_3<m_*$ and $i_4=m_*$.

To show this, we define an $(m_*-1)$-partite graph $G$ with vertex classes $W_i=\cP^{im_*}$ for~$1\le i<m_*$ by demanding that the restriction of $G$ to $W_i\cup W_j$ be isomorphic to~$Q^{m_*}_{ij}$ whenever $1\le i<j<m_*$. Notice that for such $i$ and $j$ the triple $\{i, j, m_*\}$ is red, whence~$(**)$ implies
\[
\sum_{P\in W_i} d^{\phantom{!}2}_{W_j}(P)\ge \bl\tfrac 14+\eps\br |W_{i}|\,|W_{j}|^2\,.   
\]
As we could have chosen $m_*$ so large that the conclusion of Theorem~\ref{sync} applies to $m_*-1$ and~$\eps$ here in place of $m$ and  
$\eps$ there, 
we may assume that $G$ contains a triangle, say with vertices $P^{14}\in W_{i_1}$, 
$P^{24}\in W_{i_2}$, and $P^{34}\in W_{i_3}$, where $i_1<i_2<i_3$. Now, for example, 
$P^{14}P^{24}$ being an edge of $G$ and hence of $Q^{m_*}_{i_1i_2}$ means that there is some
vertex $P^{12}\in \cP^{i_1i_2}$ such that the triple $P^{12}P^{14}P^{24}$ appears in 
$\cA^{i_1i_2m_*}$. For the same reason, the desired vertices $P^{13}$ and~$P^{23}$ exist.
Thereby Lemma~\ref{K4-modif} and, hence, Theorem~\ref{K4-} is proved.
\end{proof}

\section{Concluding remarks} 
\label{sec:conc_rem}
\subsection{Tur\'an densities of cliques in weakly quasirandom hypergraphs}
Our main result, Theorem~\ref{K4-}, asserts that the weakly quasirandom Tur\'an density 
of $K_4^{(3)-}$ is $1/4$ and many open questions remain. It would be very interesting to
determine $\pivvv(K_4^{(3)})$  or more generally  $\pivvv(K_k^{(3)})$ for arbitrary $k\geq 4$.
We recall a random construction from~\cite{Ro86} which shows that 
\begin{equation}\label{eq:wKk-lb}
	\pivvv(K_k^{(3)})\geq \frac{k-3}{k-2}\,.
\end{equation}
This lower bound is established by considering a random
$(k-2)$-colouring $\phi$ of the pairs~$[n]^{(2)}$, where the colour of each pair is chosen uniformly and 
independently among all $k-2$ colours. Given such a colouring $\phi$ we let $H_{\phi}$ be the 
$3$-uniform hypergraph with vertex set $[n]$ containing only those hyperedges $\{x,	y,z\}$ with 
$1\leq x<y<z\leq n$ that satisfy $\phi(x,y)\neq \phi(x,z)$. One can check that for any fixed $\eta>0$ with high probability 
the hypergraph $H_{\phi}$ is $(\tfrac{k-3}{k-2},\eta)$-quasirandom  for sufficiently 
large~$n$. On the other hand, for any~$k$ vertices $1\leq x_1\leq\dots\leq x_k\leq n$ two of the $k-1$ pairs $\{x_1,x_i\}$
with $i=2,\dots,k$ containing~$x_1$ must have the same colour in $\phi$. Consequently, $x_1,\dots,x_k$ cannot span a clique
and~\eqref{eq:wKk-lb} follows. We believe this construction is optimal for $k=4$ and put forward 
the following conjecture.
\begin{conj}\label{conj1}
	We have $\pivvv(K_4^{(3)})= \frac{1}{2}$.
\end{conj}
In~\cite{RRS-c} we establish a weaker version of Conjecture~\ref{conj1}.
This version is based on the following strengthened form of the assumed quasirandom condition.
\begin{dfn}
\label{12qr}
A $3$-uniform hypergraph $H=(V, E)$ is \emph{$(d, \eta, \ev)$-quasirandom}
if for every subset $U\subseteq V$ of vertices and every $X\subseteq V^{(2)}$ set of pairs of $V$
the number $e(U,X)$ of ordered pairs $(u,\{x,x'\})$ satisfying $\{u,x,x'\}\in E$,  $u\in U$, and $\{x,x'\}\in X$
satisfies
\[
	\big|e(U,X)-d|U||X|\big|\leq\eta\,n^3
\]
and we denote by $\ccQ^{(3)}_{\ev}(d,\eta)$ the class of $(d, \eta, \ev)$-quasirandom $3$-uniform hypergraphs.
\end{dfn}
With this definition at hand we define for a $3$-uniform hypergraph $F$ the corresponding quasirandom 
Tur\'an density
\begin{multline*}
	\piev(F)=\sup\bigl\{d\in[0,1]\colon \text{for every $\eta>0$ and $n\in \NN$ there exists an $F$-free,}\\
		\text{$3$-uniform hypergraph $H\in\ccQ_{\ev}(d,\eta)$ with $|V(H)|\geq n$}\bigr\}\,.
\end{multline*}
One can check that for every $k\geq 3$ with high probability 
the hypergraph $H_{\phi}$ defined by 
a random $(k-2)$-colouring $\phi$ above is indeed quasirandom in the sense of Definition~\ref{12qr}, i.e., 
it is 
$(\tfrac{k-3}{k-2},\eta,\ev)$-quasirandom 
for any fixed $\eta>0$ for 
sufficiently large $n$. Consequently, we also have
\begin{equation}\label{eq:12Kk-lb}
	\piev(K_k^{(3)})\geq \frac{k-3}{k-2}\,.
\end{equation}
%and a positive resolution of Conjecture~\ref{conj1} would imply also here equality. 
In~\cite{RRS-c} we establish a matching upper bound for $k=4$ by a proof based on the regularity method for hypergraphs. 
\begin{thm}
	We have $\piev(K_4^{(3)})=\frac{1}{2}$.
\end{thm}

Also for $k>4$ it might be possible that the lower bound given in~\eqref{eq:wKk-lb} (and~\eqref{eq:12Kk-lb})
is best 
possible and we are not aware of any better constructions. However, we remark that 
for $k=6$ there is another construction attaining the same bound. For that we consider a
random two colouring of $[n]^{(2)}$ and let $H$ consist of all triples $\{x,y,z\}$
for which the three pairs $\{x,y\}$, $\{x,z\}$, and $\{y,x\}$ are not all of the same colour. 
Again it is easy to check that  with high probability the hypergraph 
$H$ is $(3/4,\eta)$-quasirandom for every fixed $\eta>0$, while the simplest instance of Ramsey's 
theorem, the so called  ``three in a party of six theorem'', yields that $H$ is $K_6^{(3)}$-free.
It would be intriguing in case both of these constructions would be best possible.

\subsection{Hypergraph with vanishing weakly quasirandom Tur\'an density}
For the classical Tur\'an density $\pi(\cdot)$ Erd\H os~\cite{Er64} characterised
all hypergraphs $F$ with $\pi(F)=0$. Restricting the discussion to 
$3$-uniform hypergraphs, he showed that $\pi(F)=0$ if and only if $F$ is tripartite, i.e., 
$V(F)$ can be partitioned into three classes
such that every hyperedge of~$F$ contains precisely one vertex from each class. 
Since large, complete, and balanced tripartite $3$-uniform hypergraphs have density approaching $2/9$
Erd\H os deduced that if $\pi(F)>0$, then~$\pi(F)\geq 2/9$.

We establish a similar characterisation of $\{F\colon \pivvv(F)=0\}$. Clearly, this 
set contains all tripartite hypergraphs 
and the additional quasirandomness assumption considered here 
enriches this set. In fact, it follows from the 
work in~\cite{KNRS} that in addition to all tripartite hypergraphs 
it contains all
linear $3$-uniform hypergraph~$F$, where we say a hypergraph $F$ is \emph{linear}, 
if any pair of hyperedges shares at most one vertex.
In~\cite{RRS-zero} we obtain the following characterisation of hypergraphs with vanishing weakly quasirandom Tur\'an density.

\begin{thm}\label{zero}
For each $3$-uniform hypergraph $F$, the following are equivalent
\begin{enumerate}[label=\alabel]
\item \label{zero:a}$\pivvv{(F)}=0$.
\item \label{zero:b} There is an enumeration of the vertices of $F$ as $v_1, \ldots, v_f$ together with a colouring of the pairs of vertices of $F$ using the colours red, blue and green such that if for $i<j<k$ the triple $\{v_i,v_j,v_k\}$ is a 
hyperedge of $F$, then $\{v_i,v_j\}$ is red, $\{v_i,v_k\}$ is blue, and $\{v_j,v_k\}$ is green.
\end{enumerate}
\end{thm}

Theorem~\ref{zero} has the following consequence, which shows that $\pivvv$ ``jumps'' from~$0$ to at least $1/27$. 
\begin{cor}\label{jump}
If a $3$-uniform hypergraph $F$ satisfies $\pivvv{(F)}>0$, then $\pivvv{(F)}\ge\tfrac1{27}$.
\end{cor}

For the proof of Corollary~\ref{jump} we will display a weakly quasirandom hypergraph $H$
of density $1/27$, which only contains subhypergraphs satisfying condition~\ref{zero:b} 
of Theorem~\ref{zero} (in fact, it will be  universal for all such hypergraphs). Consequently, 
if $\pivvv(F)>0$, then by Theorem~\ref{zero} the hypergraph~$F$ fails to satisfy condition~\ref{zero:b}, whence it 
is not contained in~$H$ and, therefore, we have $\pivvv(F)\geq 1/27$.

The hypergraph $H$ will given by the following random construction: 
We consider a random three-colouring $\psi\colon [n]^{(2)}\to\{\textrm{red}\,,\textrm{blue}\,,\textrm{green}\}$ 
of the pairs of the first $n$ positive integers. For a given colouring $\psi$ we define the $3$-uniform 
hypergraph $H=H_{\psi}$ on the vertex set $[n]$, where we include the triple $\{i,j,k\}$ with $1\leq i<j<k\leq n$ 
in $E(H_\psi)$ if~$\psi(i,j)$ is $\textrm{red}$, $\psi(i,k)$ is $\textrm{blue}$, and $\psi(j,k)$ is $\textrm{green}$.
It follows that for any $\eta>0$  with high probability the random hypergraph $H_{\psi}$ is 
weakly $(1/27,\eta)$-quasirandom for sufficiently large $n$. Moreover, it follows from the 
construction that every subhypergraph of $H_\psi$ 
satisfies condition~\ref{zero:b} of Theorem~\ref{zero} and, hence, Corollary~\ref{jump} follows from Theorem~\ref{zero}.

We also note that if  $F$ satisfies $\pivvv(F)=0$, then by definition of $\pivvv$ the hypergraph $F$ is contained in any 
weakly quasirandom hypergraph of positive density and, in particular, $F\subseteq H_{\psi}$.
Hence, $\psi$ restricted to the vertices of this copy of~$F$
shows that~$F$  satisfies~\ref{zero:b} of Theorem~\ref{zero}, which establishes the 
implication~\ref{zero:a} $\Longrightarrow$~\ref{zero:b} of Theorem~\ref{zero}. 
The other implication is the main part in Theorem~\ref{zero} and is based on the regularity method for hypergraphs
and is the main result in~\cite{RRS-zero}.

\subsection{Two extensions of Theorem~\ref{K4-}}
We may suggest two extensions of the main result. Theorem~\ref{K4-} concerns the weakly quasirandom Tur\'an density for 
the hypergraphs~$K_4^{(3)-}$. This hypergraph consists of one apex vertex~$a$ whose link graph, i.e., the set of pairs that  
together with $a$ form a hyperedge in $K_4^{(3)-}$, is a triangle. It would be interesting 
the study the case, when the triangle is replaced by a larger clique. We discuss partial results addressing this problem 
in Section~\ref{sec:cr-cliquelink}.

For the second extension of Theorem~\ref{K4-} we consider $K_4^{(3)-}$ as a $3$-uniform hypergraph 
with three hyperedges on four vertices and, similarly, for $r\geq 3$ we may consider $r$-uniform hypergraphs with three edges on 
$(r+1)$-vertices. In fact, we established the quasirandom Tur\'an density for these hypergraphs, if the 
quasirandomness of the $r$-uniform hyperedges of~$H$ are distributed quasirandomly with respect to the 
the $(r-2)$-tuples of the vertex set (see Section~\ref{sec:cr-r}).

\subsubsection{Extending graph cliques to hypergraphs}\label{sec:cr-cliquelink}
We consider the following star-like $3$-uniform hypergraphs $S_k$. For $k\geq 3$ the hypergraph 
$S_k$ has vertex set $\{a,b_1,\dots,b_k\}$ and for all~$\binom{k}{2}$ pairs $1\leq i<j\leq k$ the triple 
$\{a,b_i,b_j\}$ is a hyperedge of $S_k$. We refer to the vertex $a$ which is contained in every hyperedge 
of $S_k$ as the \emph{apex vertex}. Clearly, for $k=3$ we have $S_3=K_4^{(3)-}$ and by Theorem~\ref{K4-}
we have $\pivvv(S_3)=1/4$. From this point of view the natural question asking to determine $\pivvv(S_k)$
for $k>3$ arises. We have partial results in this direction and begin our discussion with $S_4$.
\begin{thm}\label{S4}
	We have $\tfrac{1}{3}\leq \pivvv(S_4)\leq \tfrac{4}{9}$.
\end{thm}
The upper bound can be proved along the lines of Theorem~\ref{K4-} by using 
Theorem~\ref{thm:Fsync} for~$K_4$ instead of Theorem~\ref{sync}.

The lower bound is given by the following construction. Again we consider a
random three-colouring $\psi\colon [n]^{(2)}\to\{\textrm{red}\,,\textrm{blue}\,,\textrm{green}\}$
of the pairs of the first $n$ positive integers. Given such a colouring $\psi$ we define a $3$-uniform hypergraph
$H=H_{\psi}$ on the vertex set $[n]$ containing those hyperedges $\{x,y,z\}$ with $x<y<z$ 
where the colour pattern of the three pairs $\{x,y\}$, $\{x,z\}$, and $\{y,z\}$ satisfies
\begin{enumerate}[label=\rmlabel]
	\item\label{it:S4i} $\psi(x,y)=\psi(y,z)\neq \psi(x,z)$, or
	\item\label{it:S4ii} the ordered colour pattern $(\psi(x,y)\,,\psi(x,z)\,,\psi(y,z))$ is 
		one of the three \emph{rainbow} patterns $(\textrm{red}\,,\textrm{blue}\,,\textrm{green})$,
		$(\textrm{green}\,,\textrm{red}\,,\textrm{blue})$, or $(\textrm{blue}\,,\textrm{green}\,,\textrm{red})$.
\end{enumerate}
Note that there are six patterns of the first kind and so in total for the hyperedges 
of~$H$ we allow
nine of the $27$ possible combinations. 
Standard probabilistic tail estimates
show for any $\eta>0$ that with high probability $H$ is weakly $(1/3,\eta)$-quasirandom 
provided $n$ is sufficiently large.

It is left to show that $H$ contains no copy of $S_4$. Supposing to the contrary, 
let $a\in[n]$ be the apex vertex of a copy of~$S_4$ in~$H$ and consider its monochromatic 
neighbourhoods with respect to $\psi$ , i.e., we set
\[
	N^<_{\textrm{red}}(a)=\{x<a\colon \psi(x,a)=\textrm{red}\}
	\qand
	N^>_{\textrm{red}}(a)=\{x>a\colon \psi(a,x)=\textrm{red}\}
\]
and similarly let $N^<_{\textrm{green}}(a)$, $N^>_{\textrm{green}}(a)$, $N^<_{\textrm{blue}}(a)$, and $N^>_{\textrm{blue}}(a)$
be defined for the other two colours. By definition these six sets partition the set $V_a=[n]\setminus\{a\}$.
We consider the \emph{link graph} $L_a$ of $a$ with vertex set $V_a$ where $\{u,v\}$ forms an edge if $\{a,u,v\}$ 
is a hyperedge of~$H$. Note that due to the allowed colour patterns in~\ref{it:S4i} 
and~\ref{it:S4ii} the six neighbourhood sets are independent sets in $L_a$. Moreover, one can check
that the three sets 
\[
	N^<_{\textrm{red}}(a)\dcup N^>_{\textrm{blue}}(a)\,,\quad
	N^<_{\textrm{blue}}(a)\dcup N^>_{\textrm{green}}(a)\,,\qand
	N^<_{\textrm{green}}(a)\dcup N^>_{\textrm{red}}(a)
\]
are also independent sets and partition $V_a$. In other words, the link graph $L_a$ is 3-colourable and, hence, it cannot
contain a copy of $K_4$. In particular, the vertex $a$ cannot be the apex vertex of a copy of $S_4$ in $H$.

For general $k$ we can prove 
\begin{equation}\label{Sk}
	\frac{k^2-5k+7}{(k-1)^2}
	\leq 
	\pivvv(S_k)
	\leq
	\left(\frac{k-2}{k-1}\right)^2\,.
\end{equation}
The upper bound follows like the proof for $k=4$ along the lines of Theorem~\ref{K4-}
with the generalisation of Theorem~\ref{sync} for the clique $K_k$ (see Theorem~\ref{thm:Fsync}).

For the lower bound we consider a random $(k-1)$-colouring of $\psi\colon [n]^{(2)}\to\{0,\dots,k-2\}$.
Similar as before the colour pattern we see on the pairs of three vertices $x<y<z$ determines if 
this triple forms a hyperedge of $H$. In the general case we allow the following patterns 
\begin{enumerate}[label=\rmlabel]
	\item\label{it:Ski} $\psi(x,y)=\psi(y,z)\neq\psi(x,z)$, or
	\item\label{it:Skii} the ordered colour pattern $(\psi(x,y)\,,\psi(x,z)\,,\psi(y,z))$ is 
		rainbow (i.e., all three colours are different), but not of the form
		$(i,j,i+1)$ for $i=0,\dots,k-2$ and $j\not\in\{i, i+1\}$, where addition is taken modulo $k-1$.
\end{enumerate}
This way of all different $(k-1)^3$ patterns we allow $(k-1)(k-2)$ patterns by part~\ref{it:Ski} of the definition and 
$(k-1)(k-2)(k-3)-(k-1)(k-3)=(k-1)(k-3)^2$ patterns 
in~\ref{it:Skii}. Hence, with high probability the hypergraph $H$ is weakly $(d,\eta)$-quasirandom 
for any fixed $\eta>0$ and 
\[
	d=\frac{(k-1)(k-2)+(k-1)(k-3)^2}{(k-1)^3}
	 %=\frac{k-2+(k-2)(k-3)-(k-3)}{(k-1)^2}
	 %=\frac{k-2+(k-3)^2}{(k-1)^2}
	 =\frac{k^2-5k+7}{(k-1)^2}\,.
\]
Moreover, as above one can show that the link graph $L_a$ of every vertex $a\in[n]$ 
is $(k-1)$-colourable and, hence, it contains no $K_k$. In fact, with the similar notation as above 
it can be checked that the sets
\[
	N^<_{i}(a)\dcup N^>_{i+1}(a)\
\]
for $i=0,\dots,k-2$ form a partition of $[n]\setminus\{a\}$ into independent sets in $L_a$.
This establishes the lower bound of~\eqref{Sk}.

\subsubsection{Three $r$-tuples on $r+1$ vertices}\label{sec:cr-r}
In their concluding remarks from~\cite{GKV}, Glebov, Kr{\'a}{\v{l}}, and Volec 
suggested an analogue of Theorem~\ref{K4-} in the context of $r$-uniform hypergraphs. 
Instead of looking at $K_4^{(3)-}$ they propose to look at the $r$-uniform hypergraph $F^{(r)}$ on $r+1$ 
vertices with $3$ edges, so that a $r$-uniform hypergraph $H$ contains $F^{(r)}$ if and only if the link of some $(r-2)$-set of vertices contains a triangle. This is perfectly suited for the natural generalization of Example~1.3 to this context.

To keep the discussion simple we stick for now to the case $r=4$. Then one may start from a random directed $3$-uniform 
hypergraph $T^{(3)}_n$ with vertex set $[n]$ in which for any $3$-element subset $[n]$ one of its two {\it cyclic orientations} has been chosen at random with probabilities~$1/2$, all of these choices being made mutually independent. Then, we consider a $4$-element set to be a hyperedge of the corresponding $4$-uniform hypergraph $H(T^{(3)}_n)$ if and only if each of its 
two-element subsets are traversed by the two triples containing it in opposite directions. So~$\{w,x,y,z\}\in E\big(H(T^{(3)}_n)\big)$ happens for example in case 
$\overrightarrow{xyz}, \overrightarrow{xwy}, \overrightarrow{xzw}, \overrightarrow{ywz}\in E(T^{(3)}_n)$.
It is not hard to show that such a hypergraph $H(T^{(3)}_n)$ is $F^{(4)}$-free. 
%In fact, this holds for any choice of such cyclic orientations and has nothing to do with randomness. 
Also, it is easily checked that this hypergraph has density $\tfrac18$ and is \emph{weakly quasirandom} (i.e., it has uniform hyperedge distribution with respect to sets of vertices). This means that in analogy with~\eqref{wqr-abc} we have for any $\eta>0$ that if $n$ is sufficiently large, then with high probability all sets $U_1$, $U_2$, $U_3$, and $U_4$ of vertices satisfy 
\[
e(U_1, U_2, U_3, U_4)=\tfrac18\,|U_1|\,|U_2|\,|U_3|\,|U_4|\pm \eta\,n^4\,,
\]
where $e(U_1, U_2, U_3, U_4)$ contains those $4$-tuples $(u_1,u_2,u_3,u_4)\in |U_1|\times|U_2|\times|U_3|\times|U_4|$ such that $\{u_1,u_2,u_3,u_4\}$ is a hyperedge 
of~$H(T^{(3)}_n)$.
This prompted the authors of~\cite{GKV} to conjecture that any weakly quasirandom $4$-uniform hypergraph $H$ with 
density $>\tfrac18$ contains a copy of $F^{(4)}$.

An interesting hypergraph described by Leader and Tan in a different context in~\cite{LeTa10}, however, shows that this is not the case, and that at least twice as much density is needed. Their construction starts from a random (graph) tournament $T_n$ on $n$ vertices as in Example~\ref{ex:tournament}. Depending on $T_n$, they define a directed $3$-uniform hypergraph $D^{(3)}_n$ by assigning the cyclic orientation to any $3$-element set $\{x,y,z\}$ in such a way that it coincides with the direction of the 
three arcs spanned by $\{x,y,z\}$ in $T_n$ either once or three times. So, e.g. if 
$\overrightarrow{xy}, \overrightarrow{yz}, \overrightarrow{zx}\in E(T_n)$, 
then 
$\overrightarrow{xyz}\in E(D^{(3)}_n)$, 
while if 
$\overrightarrow{xy}, \overrightarrow{xz}, \overrightarrow{yz}\in E(T_2)$, 
then $\overrightarrow{xzy}\in E(D^{(3)}_n)$. Now the $4$-uniform hypergraph $H(D^{(3)}_n)$ defined as in the previous paragraph but starting from $D^{(3)}_n$ rather than the random orientation $T^{(3)}_n$ is easily shown to have density about $1/4$. Moreover, it is weakly quasirandom and contains no copy of $F^{(4)}$.

In the light of this example, we propose a modification of the original question: it may be observed that the intended extremal example $H(T^{(3)}_n)$ satisfies stronger quasirandomness properties than $H(D^{(3)}_n)$ does. Notably, it behaves quasirandomly with respect to pairs, which  means that for any six graphs $G_{12}, \ldots, G_{34}$ on $[n]$ about $1/8$ of the quadruples $(x_1, x_2, x_3, x_4)$ with $\{x_1, x_2\}\in E(G_{12}), \ldots, \{x_3, x_4\}\in E(G_{34})$  satisfy $\{x_1, x_2, x_3, x_4\}\in E(H)$. One may also show directly that the hypergraph $H(D^{(3)}_n)$ lacks this property.

This may suggest that any $4$-uniform hypergraph with density $>\tfrac18$ that is quasirandom with respect to pairs in this sense does indeed contain a copy of $F^{(4)}$. More generally we show in~\cite{RRS-e} that an $r$-uniform hypergraph of density $>2^{1-r}$ that is quasirandom with respect to $(r-2)$-tuples has to contain $F^{(r)}$. The proof presented in~\cite{RRS-e}
relies on the regularity method for $r$-uniform hypergraphs and is considerably more intricate than the argument presented here.
% (For a more precise statement of this result, the curious reader is referred to [3-edges]). 

We note that the case $r=2$ of this result might be viewed as the density version of Mantel's theorem.
%, since the assumption of being quasirandom with respect to the empty set is always vacuously satisfied. 
Thus the ``three edge theorem'' in~\cite{RRS-e}
provides a common generalisation of Mantel's theorem and Theorem~\ref{K4-} to the context of $r$-uniform hypergraphs. 

\subsection*{Acknowledgement} The second author thanks 
Yoshi Kohayakawa and Endre Szemer\'edi for early discussions on the problem.

\end{document}